\documentclass[12pt]{amsart}
\newtheorem{thm}{Theorem}
\newtheorem{prop}[thm]{Proposition}
\newtheorem{cor}[thm]{Corollary}
\newtheorem{lem}[thm]{Lemma}
\theoremstyle{definition}
\newtheorem{exa}{Example}

\newtheorem{que}{Question}
\newcommand{\fF}{{\mathfrak F}}
\newcommand{\IN}{{\mathbb N}}
\newcommand{\IR}{{\mathbb R}}
\DeclareMathOperator{\id}{id}
\DeclareMathOperator{\Lip}{Lip}
\DeclareMathOperator{\spann}{span}
\title{Free Banach spaces and the approximation properties}
\author{Gilles Godefroy}
\address{Institut de Math\'ematiques de Jussieu,
4 Place Jussieu 75005 Paris}
\email{godefroy@math.jussieu.fr}
\author{Narutaka Ozawa}
\address{RIMS, Kyoto University, \mbox{606-8502} Kyoto}
\email{narutaka@kurims.kyoto-u.ac.jp}

\date{\today}
\begin{document}
\begin{abstract}
We characterize the metric spaces whose free spaces have the bounded approximation
property through a Lipschitz analogue of the local reflexivity principle.
We show that there exist compact metric spaces whose free spaces fail the
approximation property.
\end{abstract}
\maketitle
\section{Introduction.}
Let $M$ be a pointed metric space, that is, a metric space equipped with a distinguished point
denoted $0$. We denote by $\Lip_0(M)$ the Banach space of all real-valued Lipschitz functions
defined on $M$ which vanish at $0$, equipped with the natural Lipschitz norm
\[
\Vert f\Vert_L=\sup \big\{\frac{\vert f(x)-f(y)\vert}{\Vert x-y\Vert}~;~(x,y)\in M^2,~x\not=y \big\}.
\]
For all $x\in M$, the Dirac measure $\delta(x)$ defines a continuous linear form on $\Lip_0(M)$.
Equicontinuity shows that the closed unit ball of $\Lip_0(M)$ is compact for pointwise convergence
on $M$, and thus the closed linear span of $\{ \delta(x)~;~x\in M\}$ in $\Lip_0(M)^*$ is an isometric
predual of $\Lip_0(M)$. This predual is called the Arens-Eells space of $M$ in \cite{weaver},
and (when $M$ is a Banach space) the Lipschitz-free space over $M$ in \cite{gk}, denoted by $\fF(M)$.
We will use this notation, and simply call $\fF(M)$ the free space over $M$. When $M$ is separable,
the Banach space $\fF(M)$ is separable as well, since the set $\{ \delta(x)~;~x\in M\}$ equipped
with the distance induced by $\Lip_0(M)^*$ is isometric to $M$.

The free spaces over separable metric spaces $M$ constitute a fairly natural family of
separable Banach spaces, which are moreover very useful in non-linear geometry of Banach
spaces (see \cite{kalton2}). However, they are far from being well-understood at this point
and some basic questions remain unanswered. We recall that a Banach space $X$ has the approximation property (AP) if the identity $\id_X$ of $X$ is in the closure of the finite rank operators on $X$ for the topology of uniform convergence on compact sets. The $\lambda$-bounded approximation property ($\lambda$-BAP) means that there are approximating finite rank operators with norm less than $\lambda$, and the (1-BAP) is called the metric approximation property (MAP). This note is devoted to the following problem:
for which metric spaces $M$ does the space $\fF(M)$ have (AP),
or (BAP), or (MAP)?
For motivating this query, recall that real-valued Lipschitz functions defined on subsets of
metric spaces extend with the same Lipschitz constant through the usual inf-convolution formula.
However, approximation properties for free spaces are related with the existence of
{\sl linear} extension operators for Lipschitz functions defined on subsets
(see \cite{bm}, and Proposition 6 below).

It is already known that some free spaces fail AP: indeed one of the main results of \cite{gk}
asserts that if $X$ is an arbitrary Banach space and $\lambda\geq 1$, then $X$ has $\lambda$-BAP
if and only if $\fF(X)$ has $\lambda$-BAP. Since moreover any separable Banach space $X$ is
isometric to a $1$-complemented subspace of $\fF(X)$ (\cite{gk}, Theorem~3.1), it follows that
$\fF(X)$ fails AP when $X$ does.

This note provides further examples of metric spaces whose free spaces fail AP.
We show in particular that some spaces $\fF(K)$, with $K$ compact metric spaces, fail AP
although MAP holds for ``small" Cantor sets.

Section~\ref{sec2} gives a characterization of the $\lambda$-BAP for $\fF(M)$ through
weak*-approximation of Lipschitz functions from $M$ into bidual spaces, somewhat similar
to the local reflexivity principle. In section~\ref{sec3}, a method used in \cite{gk}
and localized in \cite{dl} is shown to provide the existence of compact convex sets $K$
with $\fF(K)$ failing AP. Several open questions conclude the note.

\section{Lipschitz local reflexivity}\label{sec2}

For metric spaces $M$ and $X$, we denote by
$\Lip^\lambda(M,X)$ the set of $\lambda$-Lipschitz maps from $M$ into $X$.
We assume that $M$ is separable and $X$ is complete.
Fix a dense sequence $(x_n)_n$ in $M$ and define
a metric $d$ on $\Lip^\lambda(M,X)$ by
\[
d(f,g)=\sum_{n=1}^\infty \min\{d(f(x_n),g(x_n)),\,2^{-n}\}.
\]
Then, $d$ is a complete metric on $\Lip^\lambda(M,X)$ whose
topology coincides with the pointwise convergence topology.

Let $Z$ be a Banach subspace of $Y$ and
denote the quotient map by $Q\colon Y\to Y/Z$.
We say $Z$ is an \emph{$\mathrm{M}$-ideal with an approximate unit},
or an $\mathrm{M}$-iwau in short, if
there are nets of operators $\phi_i\colon Y\to Z$
and $\psi_i\colon Y\to Y$ such that
$\phi_i(z)\to z$ for every $z\in Z$,
$Q\circ\psi_i=Q$ for all $i$,
$\phi_i+\psi_i\to\id_Y$ pointwise,
and $\|\phi_i(x)+\psi_i(y)\|\le\max\{\|x\|,\,\|y\|\}$
for all $x,y\in Y$ and $i$.
We note that $\psi_i\to0$ on $Z$ and
$\|Q(y)\|=\lim\|\psi_i(y)\|$.

\begin{exa}\label{exa:A}
Let $X$ be a separable Banach space and $X_n$ be an
increasing sequence of finite-dimensional subspaces
whose union is dense. Then, we define
\[
Y=\{ (x_n)_n \in (\prod X_n)_{\infty} ~;~
\mbox{the sequence $(x_n)_n$ is convergent in $X$}\}.
\]
Then, $Y$ is a Banach space with the MAP with the metric
surjection $Q\colon Y\to X$ given by the limit.
The subspace $\ker Q$ is an $\mathrm{M}$-iwau, with
$\phi_k((x_n)_n)=(x_1,\ldots,x_k,0,0,\ldots)$.
\end{exa}

\begin{exa}
Every closed two-sided ideal
$I$ in a C$^*$-algebra is an $\mathrm{M}$-iwau.
\end{exa}

\begin{lem}[cf.\ \cite{arveson}, Theorem~6]
Let $Z\subset Y$ be an $\mathrm{M}$-iwau and $M$ be a separable
metric space. Then, for every $\lambda\geq1$ the set
\[
\{ Q\circ f ~;~ f\in\Lip^\lambda(M,Y)\}
 \subset \Lip^\lambda(M,Y/Z)
\]
is closed under the pointwise convergence topology.
\end{lem}
\begin{proof}
Let $(f_n)_n$ be a sequence in $\Lip^\lambda(M,Y)$
such that $Q\circ f_n$ converge to $F \in \Lip^\lambda(M,Y/Z)$.
To prove that $F$ lifts, we may assume that
$d(Q\circ f_n,Q\circ f_{n+1})<2^{-n}$.
We will recursively construct $g_n$ such that
$Q\circ g_n=Q\circ f_n$ and
$d(g_n,g_{n+1})<2^{-n}$.
Then, the sequence $(g_n)_n$ converges and its limit
is a lift of $F$.
For $g_{n+1}$, we define
\[
g_{n+1,i}=\phi_i\circ g_n+\psi_i\circ f_{n+1}.
\]
Then, $g_{n+1,i}\in\Lip^\lambda(M,Y)$,
$Q\circ g_{n+1,i}=Q\circ f_{n+1}$ and
\[
\lim_i d(g_n,g_{n+1,i})
= \lim_i d(\psi_i\circ g_n,\psi_i\circ f_{n+1})
= d(Q\circ g_n,Q\circ f_{n+1}) < 2^{-n}.
\]
Thus, there is $i$ such that $g_{n+1}:=g_{n+1,i}$ works.
\end{proof}

\begin{thm}\label{thm:lr}
Let $M$ be a separable metric space and $\lambda\geq1$.
Then, the free space $\fF(M)$ has the $\lambda$-BAP
if and only if $M$ has the following property:
For any Banach space $Y$ and any
$f\in\Lip^1(M,Y^{**})$, there is a net in
$\Lip^\lambda(M,Y)$ which converges to $f$
in the pointwise-weak$^*$ topology.
\end{thm}
\begin{proof}
Suppose $\fF(M)$ has the $\lambda$-BAP, and
$f\in \Lip^1(M,Y^{**})$ is given.
Then, $f$ extends to a linear contraction $\hat{f}\colon \fF(M)\to Y^{**}$.
Since $\fF(M)$ has the $\lambda$-BAP, the local reflexivity principle
yields a net of operators $T_i\colon \fF(M)\to Y$
with norm $\le \lambda$ which weak* converges to $\hat{f}$ pointwise.
Restricting it to $M$, we obtain a desired net.

Conversely, suppose $M$ satisfies the property stated in Theorem~\ref{thm:lr}.
We apply the construction described in Example~\ref{exa:A} to
$\fF(M)$ and obtain $Q\colon Y\to\fF(M)$.
Since $Z=\ker Q$ is an $\mathrm{M}$-ideal, one has a canonical
identification $Y^{**}=Z^{**}\oplus_\infty\fF(M)^{**}$.
In particular, $M\hookrightarrow Y^{**}$ naturally.
By assumption, there is a net $f_i\in\Lip^\lambda(M,Y)$
which approximates the above inclusion.
Since $Q\circ f_i\in\Lip^{\lambda}(M,\fF(M))$ converge
to $\id_M$ in the point-weak topology,
by taking convex combinations if necessary,
we may assume that they converge in the point-norm topology.
Thus by Lemma 1, $\id_M\colon M\hookrightarrow\fF(M)$ lifts to
a function $f\in\Lip^\lambda(M,Y)$. The function $f$ extends to
$\hat{f}\colon \fF(M)\to Y$, which is a lift of $\id_{\fF(M)}$.
Since $Y$ has the MAP, $\fF(M)$ has the $\lambda$-BAP.
\end{proof}

\begin{cor}\label{cor:lr}
A separable Banach space $X$ has the ${\lambda}$-BAP
if and only if for any Banach space $Y$ and any
$f\in\Lip^1(X,Y^{**})$, there is a net in
$\Lip^\lambda(X,Y)$ which converges to $f$
in the pointwise-weak$^*$ topology.
\end{cor}
This follows immediately from Theorem~\ref{thm:lr} since $X$ has ${\lambda}$-BAP
if and only if $\fF(X)$ has this same property (\cite{gk}, Theorem~5.3).
Note that we can replace ``Lipschitz maps" by ``linear operators" in
Corollary~\ref{cor:lr} and reach the same conclusion. In this case, our argument
boils down to a method due to Ando (\cite{ando}, see \cite{hww}, section~II.2).

\section{Some free spaces failing AP}\label{sec3}

We first prove:
\begin{thm}\label{thm:compl}
Let $X$ be a separable Banach space, and let $C$ be a closed convex set containing $0$
such that $\overline{\spann}[C]=X$. Then $X$ is isometric to a $1$-complemented subspace of
$\fF(C)$.
\end{thm}
\begin{proof}
The proof relies on a modification from \cite{dl} (see Lemma~2.1 in that paper)
of the proof of (\cite{gk}, Theorem~3.1).
We first recall that since every real-valued Lipschitz map on $C$ extends a Lipschitz
map on $X$ with the same Lipschitz constant, the canonical injection from $C$ into $X$ extends
to an isometric injection from $\fF(C)$ into $\fF(X)$ (see \cite{gk}, Lemma~2.3). Thus we simply
consider $\fF(C)$ as a subspace of $\fF(X)$.

Let $(x_i)_{i\geq1}$ be a linearly independent sequence of vectors in $C/2$
such that $\overline{\spann}[\{ x_i~;~i\geq 1\}]=X$ and $\Vert x_i\Vert=2^{-i}$ for all $i$.
We let $E=\spann[\{x_i~;~i\geq 1\}]$. We denote by $H=[0,1]^{\IN}$ the Hilbert cube,
by $t=(t_j)_j$ a generic element of $H$, and by $\lambda$ the product of the Lebesgue measures
on each factor of $H$. Of course, $\lambda$ is a probability
measure on $H$. Moreover, for any $n\in \IN$, we denote $H_n=[0,1]^{{\IN}\backslash\{n\}}$
and $\lambda_n$ the similar probability measure on $H_n$.

We denote $R\colon E\rightarrow \fF(X)$
the unique linear map which satisfies for all $n\geq 1$ and all $f\in\Lip_0(X)$
\[
R(x_n)(f)=\int_{H_n}\big[f(x_n+\sum_{ j\not=n} t_j x_j)-f(\sum_{j\not=n} t_j x_j)\big]d\lambda_{n}(t).
\]
It is clear that the map $R$ actually takes its values in the subspace $\fF(C)$ of $\fF(X)$. If $f$ is
G\^ateaux-differentiable, then Fubini's theorem shows that
\[
R(x)(f)=\int_{H}\langle\{\nabla f\}(\sum_{j}t_j x_j), x\rangle\, d\lambda(t)
\]
and thus $\vert R(x)(f)\vert\leq\Vert x\Vert\,\Vert f\Vert_L$.
Since the subset of the unit ball of $\Lip_0(X)$ consisting of functions which are
G\^ateaux-differentiable is uniformly dense in this unit ball (see \cite{bl}, Corollary~6.43),
it follows that $\Vert R\Vert\leq 1$. Since $E$ is dense in $X$, the map $R$ extends to
a linear operator of norm 1 from $X$ to $\fF(C)$, which we still denote by $R$.

If $\beta$ denotes the canonical quotient map from $\fF(X)$ onto $X$ (see \cite{gk}, Lemma~2.4),
we have $\beta R=Id_X$ and thus $R(X)$ is a subspace of $\fF(C)$ isometric to $X$
and $1$-complemented by the projection $R\beta$.
\end{proof}

The main corollary of this result is the following
\begin{cor}\label{cor:ap}
There exists a compact metric space $K$ such that $\fF(K)$ fails AP.
\end{cor}
\begin{proof}
Let $X$ be a separable Banach space failing the AP. It is classical and easily seen that
there is a compact convex set $K$ containing $0$ such that $\overline{\spann}[K]=X$.
By Theorem~\ref{thm:compl}, the space $\fF(K)$ contains a complemented subspace failing AP
and thus $\fF(K)$ itself fails AP.
\end{proof}

\begin{exa}
This result emphasizes the need to decide for which metric spaces $M$ - and
in particular for which compact metric spaces - the corresponding free space has the AP.
It is well-known that MAP holds when $K$ is an interval of the real line since then $\fF(K)$
is isometric to $L^1$, and more generally if $M$ is any subset of the real line since then $\fF(M)$ is 1-complemented in $L^1$. If $C$ is a closed convex subset of the Hilbert space $\ell_2$,
then $\fF(C)$ has MAP. Indeed $C$ is a $1$-Lipschitz retract of $\ell_2$ and thus $\fF(C)$
is $1$-complemented in $\fF(\ell_2)$ which has MAP by (\cite{gk}, Theorem~5.3).
A metric space $M$ is isometric
to a subset of a metric tree $T$ if and only if $\fF(M)$ embeds isometrically into
$L^1$ (\cite{godard}). It follows from \cite{mat} that for any such $M$ the space $\fF(M)$ has BAP.
Finally, it is shown in \cite{lp} among other things that for any $n\geq 1$ the space $\fF({\bf R}^n)$ has a basis, and that $\fF(M)$ has (BAP) for any doubling metric space $M$.
\end{exa}

We now observe that ``small" Cantor sets yield to free spaces with MAP.

\begin{prop}\label{prop:rho}
Let $K$ be a compact metric space such that
there exist a sequence $(\epsilon_n)_n$ tending to $0$, a real number $\rho<1/2$
and finite $\epsilon_n$-separated subsets $N_n$ of $K$ which are $\rho\epsilon_n$-dense
in $K$, then $\fF(K)$ has MAP.
\end{prop}
\begin{proof}
It follows from (\cite{bm}, Theorem~4) that if $M$ is a separable metric space and
$(M_n)_n$ is an increasing sequence of finite subsets of $M$ whose union is dense in $M$,
then $\fF(M)$ has BAP if and only if there is a uniformly bounded sequence of linear
operators $E_n\colon \Lip(M_n)\rightarrow \Lip(M)$ such that if $R_n$ denotes the restriction
operator to $M_n$, then for every $f\in\Lip(M)$ the sequence $f_n=E_nR_n(f)$ converges
pointwise to $f$. Our assumptions imply the existence of $\lambda$-Lipschitz retractions
$P_n$ from $K$ onto $N_n$, with $\lambda=(1 - 2\rho)^{-1}$, and then
$E_n(f)=f\circ P_n$ shows that $\fF(K)$ has BAP. To conclude the proof,
we observe that in the notation of (\cite{weaver}, Definition~3.2.1), the little Lipschitz space
$\mathrm{lip}_0(K)$ uniformly separates the points in $K$ (use the characteristic functions of the balls of radius $\rho\epsilon_n$ centered at points in $N_n$), and thus by (\cite{weaver}, Theorem~3.3.3)
the space $\fF(K)$ is isometric to the dual space of $\mathrm{lip}_0(K)$. Now Grothendieck's theorem
shows that $\fF(K)$ has MAP since it is a separable dual with BAP.
\end{proof}

We refer to \cite{kalton1} for more on little Lipschitz spaces and the ``snowflaking" operation.
On the other hand, Corollary~\ref{cor:ap} provides a negative result. It should be noted that the
existence of finite nested metric spaces with no ``good" extension operator for Lipschitz functions
is known (see Lemma 10.5 in \cite{bb}). This is obtained below from Corollary 5 by abstract nonsense.
Conversely, it would be interesting to exhibit spaces failing (AP) from combinatorial considerations
on finite metric spaces.

\begin{prop}\label{prop:ce}
For any $\lambda\geq 1$, there exist a finite metric space $H_{\lambda}$ and a subset
$G_{\lambda}$ of $H_{\lambda}$ such that if $E\colon\Lip(G_{\lambda})\rightarrow\Lip(H_{\lambda})$
is a linear operator such that $RE=\id_{\Lip(G_{\lambda})}$ (where $R$ is the operator
of restriction to $G_{\lambda}$) then $\Vert E\Vert\geq\lambda$.
\end{prop}
\begin{proof} Let $K$ be a compact metric space such that $\fF(K)$ fails AP and
let $(G_n)_n$ be an increasing sequence of finite subsets of $K$ whose union is dense in $K$.
Assume that Proposition~\ref{prop:ce} fails for some $\lambda_0\in{\IR}$, and thus that
extension operators with norm bounded by $\lambda_0$ exist for all pairs $(G,H)$ of finite
metric spaces with $G\subset H$. For any given $n$, we can apply this to $(G_n, G_k)$
with $k\geq n$ and use a diagonal argument to get an operator
$E_n\colon\Lip(G_n)\rightarrow\Lip(K)$ with $R_nE_n=\id_{\Lip(G_n)}$ (where $R_n$ is
the operator of restriction to $G_n$) and $\Vert E_n\Vert\leq\lambda_0$.
The operator $E_n$ is conjugate to a projection from $\fF(K)$ onto $\fF(G_n)$,
and it follows that $\fF(K)$ has $\lambda_0$-BAP (with a sequence of projections),
contradicting our assumption on $K$.
\end{proof}

Our work leads to a number of natural questions. We conclude this note by stating some of them.
The first one is due to N. J. Kalton (see \cite{kalton3}, Problem~1):

\begin{que}\label{que:1}
Let $M$ be an arbitrary uniformly discrete metric space, that is, there exists $\theta>0$
such that $d(x,y)\geq\theta$ for all $x\not=y$ in $M$.
Does $\fF(M)$ have the BAP? Note that AP holds by (\cite{kalton1}, Proposition~4.4).
Proposition~\ref{prop:ce} shows that a simple step-by-step approach could not suffice.
A positive answer to Question~\ref{que:1} would imply that every separable Banach space $X$
is approximable, that is, the identity $\id_X$ is pointwise limit of an equi-uniformly
continuous sequence of maps with relatively compact range. Note that by (\cite{kalton3}, Theorem~4.6)
it is indeed so for $X$ and $X^*$ when $X^*$ is separable. On the other hand, a negative answer to Question~\ref{que:1} would
provide an equivalent norm on $\ell_1$ failing MAP and this would solve two classical
problems in approximation theory (\cite{casazza}, Problems~3.12 and 3.8).
\end{que}
\begin{que}
Is there a countable compact space $K$ such that $\fF(K)$ fails (AP)?
\end{que}

\begin{que}
Let $X$ be a separable Banach space.
Does there exist a compact convex subset $K$ of $X$ containing $0$ such that
$\overline{\spann}[K]=X$ and moreover $K$ is a Lipschitz retract of $X$?
Note that when it is so, $X$ has BAP if and only if $\fF(K)$ has BAP.
The answer to this question is positive when $X$ has an unconditional basis:
indeed if all the coordinates of $x\in X$ are strictly positive,
the order interval $[-x, x]=K$ works since truncation by $x$ shows that $K$ is a Lipschitz retract.
\end{que}
\begin{que}
According to \cite{gk}, Definition~5.2., a separable Banach space $X$ has
the $\lambda$-Lipschitz BAP if $\id_X$ is the pointwise limit of a sequence $F_n$
of $\lambda$-Lipschitz maps {\sl with finite-dimensional range}, and this property
is shown in \cite{gk} to be equivalent with the usual $\lambda$-BAP.
Is it possible to dispense with the assumption that the $F_n$'s have finite-dimensional
range and still reach the conclusion? Corollary~\ref{cor:ap} suggests that this improvement
should not be straightforward.
\end{que}

\subsection*{Acknowledgments}
This work was initiated during the Concentration Week on "Non-Linear Geometry
of Banach Spaces, Geometric Group Theory, and Differentiability" organized
in College Station (Texas) in August 2011.
We are grateful to W.B.~Johnson, F.~Baudier, P.~Nowak and B.~Sari
for the perfect organization and stimulating atmosphere of this meeting.
We are also grateful to G.~Lancien for useful conversations.
The second author was partially supported by JSPS and the Sumitomo Foundation.


\begin{thebibliography}{HWW}
%
\bibitem[An]{ando} T. Ando:
{\sl A theorem on nonempty intersection of convex sets and its application}.
J. Approximation Theory \textbf{13} (1975), 158--166.
%
\bibitem[Ar]{arveson} W. Arveson:
{\sl Notes on extensions of C*-algebras}.
Duke Math. J. \textbf{44} (1977), 329--355.
%
\bibitem[BL]{bl} Y. Benyamini, J. Lindenstrauss:
{\sl Geometric nonlinear functional analysis}. Vol. 1.
American Mathematical Society Colloquium Publications 48. American Mathematical Society, 2000. xii+488 pp.
%
\bibitem[BB]{bb} A. Brudnyi, Y. Brudnyi:
{\sl Metric spaces with linear extensions preserving Lipschitz condition},
Amer. J. Math. \textbf{129, 1}(2007), 217-314.

%
\bibitem[Bo]{bm} L. Borel-Mathurin:
{\sl Approximation properties and non-linear geometry of Banach spaces}.
Houston J. of Math., to appear.
%
\bibitem[Ca]{casazza} P. Casazza:
{\sl Approximation properties}, Chapter 7, pp. 271--316
in {\sl Handbook of the geometry of Banach spaces, Volume 1},
W. B. Johnson and J. Lindenstrauss, editors, Elsevier, Amsterdam 2001.
%
\bibitem[DL]{dl} Y. Dutrieux, G. Lancien:
{\sl Isometric embeddings of compact spaces into Banach spaces}.
J. Funct. Anal. \textbf{255} (2008), 494--501.
%
\bibitem[Go]{godard} A. Godard:
{\sl Tree metrics and the Lipschitz-free spaces}.
Proc. Amer. Math. Soc. \textbf{138} (2010), 4311--4320.
%
\bibitem[GK]{gk} G. Godefroy, N. J. Kalton:
{\sl Lipschitz-free Banach spaces}.
Studia Math. \textbf{159} (2003), 121-141.
%
\bibitem[HWW]{hww} P. Harmand, D. Werner, W. Werner:
{\sl M-ideals in Banach spaces and Banach algebras},
Lecture notes in Math. 1547, Springer-Verlag 1993.
%
\bibitem[Ka1]{kalton1} N. J. Kalton:
{\sl Spaces of Lipschitz and H\"older functions and their applications}.
Collect. Math. \textbf{55} (2004), 171--217.
%
\bibitem[Ka2]{kalton2} N. J. Kalton:
{\sl The nonlinear geometry of Banach spaces}.
Rev. Mat. Complut. \textbf{21} (2008), 7--60.
%
\bibitem[Ka3]{kalton3} N. J. Kalton:
{\sl The uniform structure of Banach spaces}.
Math. Ann. To appear.
%
\bibitem[LP]{lp} G. Lancien, E. Pernecka:
{\sl Approximation properties and Schauder decompositions in Lipschitz-free spaces},
to appear.
%
\bibitem[Mat]{mat} J. Matousek:
{\sl Extension of Lipschitz mappings on metric trees},
Comment. Math. Univ. Carolinae \textbf{31, 1} (1990), 99-104.
%
\bibitem[We]{weaver} N. Weaver:
{\sl Lipschitz Algebras}. World Scientific Publishing Co., Inc., River Edge, NJ, 1999. xiv+223 pp.
%
\end{thebibliography}
\end{document}